\documentclass[11pt,leqno]{article}

\usepackage[all]{xy}
\usepackage{amssymb,amsmath,amsthm,    url,rotating}
\usepackage{mathrsfs}
  \usepackage{graphicx,epsfig}
  \usepackage{xcolor,graphicx}
\usepackage{hyperref}
    \usepackage{makeidx}
\newcommand{\n}{\noindent}

\newcommand{\vp}{\varepsilon}
\newcommand{\bb}[1]{\mathbb{#1}}
\newcommand{\cl}[1]{\mathcal{#1}}

\newcommand{\ovl}{\overline}

\theoremstyle{plain}
\newtheorem{thm}{Theorem}[section]
\newtheorem{lem}[thm]{Lemma}

\newtheorem{pro}[thm]{Proposition}

\theoremstyle{definition}

\newtheorem{dfn}[thm]{Definition}

\theoremstyle{remark}
\newtheorem{rem}[thm]{Remark}

\numberwithin{equation}{section}

\def\tilde{\widetilde}

\renewcommand{\tilde}{\widetilde}

\def\C{\bb  C}
\def\CC{\bb  C}

\def\F{\bb  F}

\def\d{\delta}

\def\CC{\bb  C}

\def\CC{\bb  C}

\def\F{\bb  F}

 \def\D{\Delta}
\def\d{\delta}

\def\CC{\bb  C}

\def\phi{\varphi}

\def\n{\noindent}
\def\nl{\nolimits}

\begin{document}

\def\C{\mathscr{C}}
\def\B{\mathscr{B}}
\def\I{\cl  I}
\def\e{\cl  E}
 
  \def\a{\alpha}
 
       \def\t{\theta}

   \title{On   $C^*$-algebras  with   Local     Lifting  Property
and Weak Expectation Property}

\author{by\\
Gilles  Pisier\footnote{ORCID    0000-0002-3091-2049}   \\
Sorbonne Universit\'e\\
and\\ Texas  A\&M  University}

\def\C{\mathscr{C}}
\def\B{\mathscr{B}}
\def\I{\cl  I}
\def\e{\cl  E}
 
  \def\a{\alpha}
 
       \def\t{\theta}
       
  \maketitle
  
\begin{abstract}   
We   give a new, somewhat simpler,  presentation of
                the author's recent construction of a non-nuclear $C^*$-algebra which has
                both the local lifting property (LLP) and the weak expectation property (WEP).
               
   \end{abstract}
  
  
    \medskip

  \def\L{\cl L}
  
  \section{Introduction}
  The main goal of this note is to give a hopefully  simpler 
  proof that there exist non-nuclear $C^*$-algebras with
                both the local lifting property (LLP) and the weak expectation property (WEP).
                (See \cite{158} for   operator spaces with analogous properties).
                We will obtain this goal in Theorem \ref{ty'}.  While this paper is mostly expository, the new approach produces some ``new" results, e.g. in Remarks
 \ref{ju14} and  \ref{ju15}.
                
  We start by an outline of  
  our main results.  
  
    We first introduce some basic terminology and notation.\\
  Let $A,B$ be $C^*$-algebras and let $E\subset A$ be an operator subspace.
                 Given a $C^*$-algebra $D$,  a linear map
                 $u: E \to B$ will be called $D$-nuclear  with constant $c$ if
                 $$\| Id_D \otimes u:  D\otimes_{\min} E \to D \otimes_{\max} B\|\le c.$$
                 If this holds with $c=1$ we say that $u$ is $D$-nuclear.
                 
                 This terminology is   inspired by the well known definition of
               nuclear  $C^*$-algebras:  $A$ is nuclear if the identity
               on $A$ is $D$-nuclear for all $D$.
  
  We first need some basic notation.\\
  For any $C^*$-algebra $C$ we denote
  $$ \cl L(C)=
                \ell_\infty(C)/c_0(C).$$
    For short we set $L=  \ell_\infty(C)$ and $\cl L= \cl L(C)$.          Let $Q: L \to \cl L $ denote the quotient map.
                For any $x=(x_n)\in L$ we have
                $\|Q(x)\|=\limsup\nl_{n\to \infty}\|x_n\|.$\\
  Our basic starting point is an isometric $*$-homomorphism              
                $$F: C \to \L$$
                together with a completely contractive self-adjoint lifting  
                 $f=(f_n): C \to L$. Lifting means here as usual
                 that $F=Qf$ and self-adjoint means $f(x^*)=f(x)^*$ for all $x\in C$.           
                 The existence of such an $f$ is of course a nontrivial assumption,
                 but it is guaranteed to hold if $C$ has the LP.
                 \footnote{Incidentally, when $C$ is separable, it is open whether $F$ always admits
                 a contractive lifting.}
                 Unfortunately we need to impose a more technical   condition,
                 as follows: 
                 \begin{equation}\label{cond}  
                  \text{For any f.d.}\ E \subset C \  \ \  \limsup\|{f^{-1}_n}_{|f_n(E)} \|_{cb}\le 1.
                 \end{equation}
                 This means that the coordinates of the  lifting $f=(f_n)$ are asymptotically 
                 locally almost completely isometric.
                 
                 We will associate to $F$ two $C^*$-algebras $Z\subset L$
                 and $A\subset \cl L$.
                 Let $E_n\subset C$ be f.d.s.a. subspaces
                 and let $(m(n))$ be a   sequence of integers.
                 We set
                  $$T_n={f_{m(n)}}_{|  E_{n-1} }: E_{n-1} \to E_n,$$
                  assuming   $\{E_n\}$  adjusted so that
              
                 $$\forall n\ge 0\qquad T_n(  E_{n-1})+T_n(  E_{n-1}) T_n(  E_{n-1})  \subset  E_n.$$
                 For any $k> n$ let $T_{[k,n]}=T_kT_{k-1}\cdots T_n: E_{n-1} \to E_k$ and hence also $T_{[n,n]}=T_n$.
                
 Let $Z_n\subset L$ be the set formed by
 all $x=(x_k)\in L$ such that $x_n\in E_n$      
 and $x_k=T_{[k,n+1]} (x_n)$ for all $k>n$. We set
 $$Z=\ovl{\cup_n Z_n} \subset L.$$
 Clearly $x_k=0$ for all $k\ge N$ implies $x\in Z_N$ and hence $c_0(C)\subset Z$. Moreover,
 it is easy to see that
 $Z_n     \subset  Z_{n+1}$. If each  map
  $T_{[k,n]}$  was the restriction of a $*$-homomorphism it would follow easily that
 $Z_n    Z_n  \subset  Z_{n+1}$. 
 But in our setting the $T_n$'s will only be approximately multiplicative
 and if they are so  at a sufficiently fast rate we still find
that $Z=\ovl{\cup Z_n}  $  is a $C^*$-subalgebra of $L$.     
  We then define the $C^*$-algebra $A=Q(Z)$ or equivalently   $$A= Z/c_0(C).$$              
 
 \begin{thm}\label{tchy} If $C$ has the lifting property (LP in short), 
 and if \eqref{cond} holds,  the spaces $(E_n)$ and the sequence $(m(n))$ can be selected so that $Z$ has the LP and $A$ the LLP.
 \\ Moreover, if  $D$ is a separable $C^*$-algebra for which $F: C \to \L$ is $D$-nuclear, they can be selected so that 
 in addition $A$ is   $D$-nuclear.
 \end{thm}
 
 The proof of this theorem appears as a combination of statements in \S \ref{am}.
 See Propositions \ref{pch2} (construction of the system),  \ref{pch2'} ($Z$ has the LP), \ref{pch3} ($A$ has the LLP) and Theorem \ref{ty} (about $D$-nuclearity).  
 
 Thus if we can produce an $F$ that is $\C$-nuclear (here $\C=C^*(\F_\infty)$) on a
 $C$ with the lifting property (LP in short) then we will obtain a $C^*$-algebra $A$ 
 with both WEP and LLP. Indeed, in Kirchberg's
 viewpoint $A$ has  WEP  (resp. the  LLP) iff $Id_A$ is $\C$-nuclear (resp. $\B$-nuclear, where $\B=B(\ell_2)$). In addition if $C$ is non-exact and  we can arrange 
 (using the condition \eqref{cond}) for the finite dimensional obstructions to the exactness of $C$ to be still present in $A$,  this will guarantee that
  $A$ is not exact, and hence not nuclear. \\
  The obvious question is then: for $D=\C$ or $\B$,  how to produce nice examples of $D$-nuclear
  embeddings $F: C \to \L$  satisfying \eqref{cond} ?
 This is provided by the setting of  cone algebras, see \S \ref{ca}.

 In our initial paper \cite{155} we used only the existence of a local lifting for our $F$,
 here the use of a global lifting simplifies somewhat the  exposition.
 In particular, we will use freely the following characterization of the
 LP from \cite{157} (see also \cite{160}).
 
 \begin{pro}\label{ppchy} 
 Let $C$ be a $C^*$-algebra with LP.
 For any f.d. subspace $E\subset C$ there is $t^E$ in the unit ball
 of $\C \otimes_{\max} E$ (as defined below) such that
 for any other $t$ in the same unit ball there is a unital $*$-homomorphism
 $\pi_t: \C \to \C$ such that $[\pi_t\otimes Id_E] (t^E)=t$.\\
 \end{pro}
  
  Let  $C,B$ be $C^*$-algebras and let $E\subset C$ and $F\subset B$ be   operator subspaces.
  We denote by $E\otimes F$   the algebraic tensor product.
  We denote (somewhat abusively) by $E\otimes_{\max} F$
  the closure of $E\otimes  F$ in $C\otimes_{\max} B$.
  We should emphasize that $E\otimes_{\max} F$ depends strongly
  on the ambient $C^*$-algebras  $C,B$.
  We 
  denote by $MB(E, F)$ the set of  (linear) maps $u: E \to F$ 
  such that $Id_\C \otimes u$ is bounded from  $\C\otimes E$ equipped
  with the norm induced by $\C\otimes_{\max} C$ to $\C\otimes F$
  equipped
  with the norm induced by $\C\otimes_{\max} B$. 
   When this holds,  we 
  define
   \begin{equation}\label{jr1} \|u\|_{mb}=\|Id_\C \otimes u: \C\otimes_{\max} E \to \C\otimes_{\max} F\|, \end{equation} 
  and we equip the space $MB(E, F)$ with this norm.\\
   If $C$ has the LP, it is easy to check with Proposition \ref{ppchy} that   
    \begin{equation}\label{tE}
    \|u\|_{mb} = \|[Id_\C \otimes u](t^E) \|_{\max}.
   \end{equation}
    See Lemma \ref{ch1} for full details.\\
   Since any   $C^*$-algebra $D$ is a quotient of the $C^*$-algebra of some free group,
   for any such $D$ we have (see \cite[Prop. 1.1]{157} for more details)
  \begin{equation}\label{tE'}
  \|Id_D \otimes u:D\otimes_{\max} E \to D\otimes_{\max} F\| \le \|u\|_{mb}.\end{equation}

  {\bf Some abbreviations:} We use c.b. (resp. c.c.) for completely bounded
  (resp. completely contractive), c.p. (resp. u.c.p.) for completely positive (resp. unital completely positive), and c.i. for
  completely isometric. Note that a ``contractive" operator or a ``contraction"
  is an operator of norm less than or equal to one. Similarly ``completely contractive" or ``complete contraction"   means the
  cb-norm is less than or equal to one. We also abbreviate self-adjoint by s.a. 
  
   {\bf Some background:} 
   A $C^*$-algebra $C$ has the lifting property (LP) if any c.p.c.c. map into a quotient $C^*$-algebra admits a c.p.c.c. lifting. 	It is known that this holds for $A$ if and only if it does for the unitization of $A$.
   A unital $C^*$-algebra $C$ has the lifting property (LP) if and only if any u.c.p. map into a quotient $C^*$-algebra admits a u.c.p. lifting. 
  Nuclear $C^*$-algebras have the LP (due to Choi-Effros)  
  but also $\C=C^*(\F_\infty)$ has the LP  (due to Kirchberg).   See \cite{157} and \cite{ES}
   for more background, precise references and more recent results on the LP.	 
   
  \section{Inductive systems}\label{is}
  
  For lack of valuable examples, we cannot 
  use for our lifting  $f$ a bona fide $*$-homomorphism,
but  using instead an approximate $*$-homomorphism
(or what we call an $\vp$-morphism) in \S \ref{ca}
 we will find examples in the framework of cone algebras
 to which Theorem \ref{tchy} can be applied
    to produce
non-nuclear $C^*$-algebras with the WEP
and the  LLP. In this section we  show by rather elementary arguments how
our inductive limits of linear spaces can lead to $C^*$-algebras.

  \begin{dfn}\label{rd2} Let $C,C_1$ be $C^*$-algebras with s.a. subspaces $E\subset C$ and $E_1\subset C_1$. Let $\vp>0$
            A linear map ${\psi}:  E \to E_1$ 
            (we will restrict in (iii) to the s.a. case for simplicity)
            will be called an $\vp$-morphism
            if  \begin{itemize}
             \item[ {\it (i)}]  $\|{\psi}\|\le 1+\vp$,
             \item[{ \it (ii)}]
        for any $x,y\in  E$ with $xy\in   E$ we have
          $\|{\psi}(xy) -{\psi}(x){\psi}(y)\|\le \vp \|x\|\|y\|,$
          \item[{ \it (iii)}]
          ${\psi}$ is self-adjoint i.e. for any $x\in   E$    we have        ${\psi}(x^*) ={\psi}(x)^*$.
           \end{itemize}
           \end{dfn}
           \noindent
           Obviously the restriction to $E$ of any $*$-homomorphism from $C$ to $C_1$
           is an $\vp$-morphism for any $\vp>0$.
           
          \begin{rem} \label{rch1} Let $E_n \subset C_n$ be s.a. subspaces of $C^*$-algebras $(n\ge 0)$.
          Let ${\psi}_n:  E_{n-1} \to E_n$ be    an $\vp_n$-morphism.
         Assuming 
          $ {\psi}_1( E_0){\psi}_1(  E_0) \subset   E_1 $,
           the composition 
          ${\psi}_2{\psi}_1:  E_0 \to E_2$
          has norm $\le (1+\vp_1)(1+\vp_2) $
          and hence satisfies {\it (i)} and {\it (ii)} with $\vp=\d_2$ for
         \begin{equation}  \label{o3'}\d_2= (1+\vp_2)\vp_1 +\vp_2 (1+\vp_1)^2
         .\end{equation}          This is immediate by the triangle inequality
          since
          $$ {\psi}_2{\psi}_1(xy)-{\psi}_2{\psi}_1(x){\psi}_2{\psi}_1(y)=
         [ {\psi}_2({\psi}_1(xy)-{\psi}_1(x){\psi}_1(y)) ]+[{\psi}_2 ({\psi}_1(x){\psi}_1(y)) -{\psi}_2{\psi}_1(x){\psi}_2{\psi}_1(y)].$$
          More generally, assuming $ {\psi}_{n}( E_{n-1}){\psi}_{n}(  E_{n-1}) \subset   E_{n } $ for all $n\ge 1$      
          then ${\psi}_{n}\cdots {\psi}_2{\psi}_1:   E_0 \to E_n$ has norm 
          $\le (1+\vp_{n})\cdots(1+\vp_2)(1+\vp_1)$.
          Let  $\pi_n= (1+\vp_n)\cdots (1+\vp_1)$.
          Repeating the argument for  \eqref{o3'}  shows that  ${\psi}_n\cdots {\psi}_2{\psi}_1$
          satisfies {\it (ii)}  with $\vp=\d_n$ for
 $\d_n$ satisfying for all $n\ge 2$ :
          $$\d_n=(1+\vp_n)\d_{n-1}  + \vp_{n}  \pi_{n-1}^2.$$
          Thus the number $\d_n'= \pi_n^{-1} \d_n$
          satisfies
          $$\d'_n= \d'_{n-1}+   \vp_{n}  \pi_{n-1}^2/\pi_n\le  \d'_{n-1}+   \vp_{n}  \pi_{n-1} .$$
          Thus if we assume that   $\sum \vp_n<\infty$ so that $\sup_n \pi_n= c<\infty$ we find for all $n\ge 2$ :
          $$\d_n'\le \d_1' +c\sum\nl_2^n \vp_{k} $$
          and hence for any $n\ge 1$ (note $\d_1=\vp_1 \ge \d_1' $)
           \begin{equation}  \label{ee3}\d_n\le \pi_n (\d_1' +c\sum\nl_2^n \vp_{k} )\le c \vp_1+c^2 \sum\nl_2^n \vp_{k}.\end{equation} 
        In conclusion, we have
             \begin{equation}  \label{ee3'}
             \d_n\le  e^2 \sum\nl_1^n \vp_{k},\end{equation} 
             and hence ${\psi}_{n}\cdots {\psi}_2{\psi}_1$ satisfies {\it (ii)}   with $\vp=e^2 \sum\nl_1^n \vp_{k}$.
             In addition, we observe that if $\sum \vp_k\le 1$ so that $c\le e$, then
              $\pi_n \le \exp{\sum \vp_k} \le 1 + (e-1) \sum \vp_k $, and hence ${\psi}_{n}\cdots {\psi}_2{\psi}_1$ satisfies {\it (i)}   with $\vp=e^2 \sum\nl_1^n \vp_{k}$.
           \end{rem}
Thus we have proved the following:

\begin{lem}\label{rch1'} Let $\psi_k: E_{k-1} \to E_k$ be $\vp_k$-morphisms with $\sum \vp_k\le 1$.
The composition ${\psi}_{n}\cdots {\psi}_2{\psi}_1:   E_0 \to E_n$ is an $\vp$-morphism
with $\vp =e^2 \sum\nl_1^n \vp_{k}$.
\end{lem}

  \begin{dfn} 
Let $C_n$ be $C^*$-algebras.
By an admissible inductive system of o.s.  
we mean a sequence $(E_n,T_n)_{n\ge 0}$    where, for each $n$, $E_n \subset C_n$
is a f.d.s.a. operator subspace, 
the map
$T_n : E_{n-1} \to E_{n}$ is a $\vp_n$-morphism  such that $T_n(E_{n-1})T_n(E_{n-1}) \subset E_{n }$ 
and  $(\vp_n)$ is  such that $$\sum \vp_n\le 1.$$
 \end{dfn}

 Let $L=\ell_\infty(\{C_n\}) $, $\cl I_0 =c_0(\{C_n\})$ 
            and $\L=L/ \cl I_0$.  Let $Q : L \to \L$ be the quotient map.

Recall that  for any  $x\in \cl L$
admitting $y=(y_n)\in L$ as a lift (i.e. $Q(y)=x$)
we have
\begin{equation}  \label{ech3}
\|x\|_\cl L = \limsup\|y_n\|.
\end{equation}  

\def\t{\theta}
Let $\t_n  :  E_n  \to L$ be
the  mapping  defined for $n\ge 0$ by   
\begin{equation} \label{ju12}  \forall  x\in  E_n\quad  \t_n(x)=(0,\cdots,  0,x,T_{n+1}(x),T_{n+2}T_{n+1}(x), T_{n+3}T_{n+2}T_{n+1}(x),\cdots)\end{equation}
where  $x$  stands  at  the   place of index  $n$. The summability of $(\vp_n)$ ensures that 
$\t_n(x)$ is a bounded sequence.\\
Note that
\begin{equation} \label{che0}
\forall x\in E_n \qquad \t_n(x)-\t_{n+1}   T_{n+1}(x)\in c_0(\{C_n\}),
\end{equation} and hence $Q \t_n= Q \t_{n+1}   T_{n+1}$.
Let 
$$Z_n=C_0\oplus \cdots \oplus C_{n-1} \oplus   \{ (x,T_{n+1}(x),T_{n+2}T_{n+1}(x), T_{n+3}T_{n+2}T_{n+1}(x),\cdots)\mid x\in E_n \}   \subset L$$ or equivalently

\begin{equation} \label{ju11} Z_n=[C_0\oplus \cdots \oplus C_{n-1} \oplus 0 \oplus 0\oplus \cdots  ] + \t_n(E_n) \subset L\end{equation}
and
$$A_n=Q(Z_n).$$
Note that $(Z_n)$ and $(A_n)$ are non decreasing sequences.
We   set
$$Z =\ovl{\cup Z_n}\subset L \text{   and   }A= \ovl{\cup A_n}\subset \L.$$

\begin{pro}\label{pch1} 
For any admissible inductive system as above,  $A$ and $Z$ are $C^*$-subalgebras.
\end{pro}

\begin{rem}\label{roct} In the unital case, if the spaces $E_n$ and the maps $T_n$ are all unital 
the resulting algebras $Z$ and $A$ are clearly unital.
\end{rem}
\begin{proof} 
  
      Note  that  the  infinite  product
$\prod_{j\ge  1}  (1+\vp_j)    $  converges.  
We  define  $\eta_n>0$  by  the  equality
$1+  \eta_n=  \prod\nl_{j\ge  n}  (1+\vp_j)$
so  that  $\eta_n  \to  0$.

A  priori 
$Z=\ovl{\cup  Z_n}  \subset  {  L}.$
   is  a  s.a.  subspace.  
To  check  that  $Z$  is  actually a  subalgebra  of  ${  L}$
we  will  show  that for all $n$ and all $a,b$ in $Z_n$ the product $   ab$ lies in $ Z$.

For $x\in E_n$ let
 \begin{equation}  \label{ju3} w_n(x) =(x, T_{n+1}(x), T_{n+2}T_{n+1}(x) , \cdots) \in \ell_\infty(\{C_n,C_{n+1}, \cdots \})\end{equation}  so that
 \begin{equation}  \label{ju4} \t_n(x)=(0,\cdots,0, w_n(x))\end{equation}
By Lemma \ref{rch1'}, the maps $w_n$ and $\t_n$ are $\vp$-morphisms
for $\vp=e^2 \sum_{k>n} \vp_k$.\\
Note that a typical element $a$ (resp. $b$) of $Z_n$ 
is of the form $a=(a_0,\cdots,a_{n-1}, w_n(a_n) ) $  
(resp. $b=(b_0,\cdots,b_{n-1}, w_n(b_n) ) $.
Then we also have
$a=(a_0,\cdots,a_{n-1}, a_n, w_{n+1}(T_{n+1} a_n) ) $  
(resp. $b=(b_0,\cdots,b_{n-1},b_n,  w_{n+1}(T_{n+1} b_n) ) $.  \\
Note $T_{n+1}(a_n) T_{n+1}(b_n)  \in E_{n+1}$.
Thus
$$d(ab, Z_{n+1})\le   \|  w_{n+1}(T_{n+1}(a_n)    )  w_{n+1}(T_{n+1}(b_n)    ) - w_{n+1}(T_{n+1}(a_n) T_{n+1}(b_n)   )  \|   $$
 and hence since $\|a_n\|\le  \|a\|  $ and $ \|b_n\| \le \|b\|$
$$d(ab, Z_{n+1})\le     e^2 (\sum\nl_{k>n+1} \vp_k) \|T_{n+1}(a_n) \| \|T_{n+1}(b_n) \|
\le  e^2 (\sum\nl_{k>n+1} \vp_k) (1+\vp_{n+1})^2 \|a\| \|b\| .$$But  now  since  $Z_n\subset  Z_{n+m}$ for any $m>0$, this also implies
$$\forall  a,b\in  Z_n\quad  d(ab,  Z_{n+m+1})\le   e^2 (\sum\nl_{k>n+m+1} \vp_k) (1+\vp_{n+m+1})^2 \|a\| \|b\|   \to  0,$$
when $m\to \infty$ and  hence  $ab\in  \ovl{\cup  Z_n}  =Z$.

 Clearly  the  same  conclusion  holds  for  any  $a,b\in  {\cup  Z_n}$,  so  that  $Z$
(which,    as  we  already  noticed,  is  s.a.)      is  a  $C^*$-subalgebra
of  ${  L}$. Since $Q(Z)$ is closed, it follows
that $A=Q(Z)$ is  a $C^*$-subalgebra of $\cl L$.  
\end{proof}

%\begin{rem} The algebra $Z$ consists of all sequences $(z_n)$ such that
%for any $\vp>0$ there is $N$ and $x\in E_N$
%such that $\limsup_{n\to \infty} \|z_n -\t_N(x)_n\|\le \vp$.
%\end{rem}

\begin{dfn}\label{jr4} An o.s. $X$ locally embeds in another one $C$ if for any $\vp>0$ and any f.d. subspace $E\subset X$
there is a subspace $F\subset C$ and an isomorphism 
$w: E \to F$ such that $\|w\|_{cb}\|w^{-1}\|_{cb} \le 1+\vp$. Moreover
we set
$$d_{cb}(E,F)=\inf \{\|w\|_{cb}\|w^{-1}\|_{cb} \}$$
where the infimum runs over all possible such $w$.\\
More generally, we say that $X$ locally embeds in a sequence $\{C_n\}$ of $C^*$-algebras
if for any $\vp>0$ and for any $E$  there is for  some $n$  a subspace $F\subset C_n$ 
such that $d_{cb}(E,F)\le 1+\vp$.
\end{dfn}

We will   need to discuss some additional properties:
\begin{equation}\label{ech50}   \| {T_n  
}     \|_{cb} \le  1+\vp_n ,
\end{equation}
and hence
\begin{equation}\label{ech50'}   \| {T_{[k,n]} } 
     \|_{cb} \le  \prod\nl_{n\le j\le k} (1+\vp_j) \le 1+\eta_n.
\end{equation}
We will also need 
\begin{equation}\label{ech5}   \| {T_n^{-1}}_{| {T_n(E_{n-1} )}
}     \|_{cb} \le  1+\vp_n.
\end{equation}
and hence
\begin{equation}\label{ech5'}   \| {T_{[k,n]}^{-1}}_{| { T_{[k,n]}(E_{n-1} ) }
}     \|_{cb} \le  \prod\nl_{n\le j\le k} (1+\vp_j) \le 1+\eta_n.
\end{equation}
 
   \begin{pro}\label{pch10}
  Assume  
that the   inductive system satisfies \eqref{ech50} and \eqref{ech5},
 then $A$ locally embeds in $\{C_n\}$. In fact for all $n\ge 0$
 \begin{equation}\label{echy5}
 d_{cb} (E_n,A_n)\le (1+\eta_{n+1})^2.
 \end{equation}
\end{pro}
\begin{proof}
By perturbation it clearly suffices to show that
$A_n$ locally embeds in $\{C_k, k\ge 0\}$ for all $n$. Fix $n$.
Recall that $Q: L \to \cl L$ denotes the quotient map.
Consider the map $\t_n : E_n \to L$
and let $\alpha_n = Q\t_n : E_n \to A_n$.
We have
$\|  \alpha_n\|_{cb} \le \| \t_n\|_{cb} \le \max\{ 1,  \sup\nl_{k> n}\|T_{[k,n+1]}\|_{cb}\}
$ and hence by  \eqref{ech50}
$\|  \alpha_n\|_{cb}  \le 1+\eta_{n+1}$.
 To estimate $\|  \alpha_n^{-1}\|_{cb} $ we will use \eqref{ech5}.
Let $x\in E_n$. We have $\|T_{[k,n+1]} (x)\| \ge (1+\eta_{n+1} )^{-1} \|x\|$ by
\eqref{ech5'} and hence
$$\|x\|=\limsup \|T_{[k,n+1]} (x)\| \ge (1+\eta_{n+1} )^{-1} \|x\| .$$
The latter remains true for any $x\in M_N(E_n)\simeq M_N \otimes E_n$ and any fixed $N$
when $T_{[k,n+1]} $ is replaced by $Id_{M_N} \otimes T_{[k,n+1]} $.
Thus we obtain
 \begin{equation}\label{cho4}
 \|  \alpha_n^{-1}\|_{cb} \le 1+\eta_{n+1} .
 \end{equation}
This implies $d_{cb}(A_n, E_n)\to 1$ when $n \to \infty$, and since
$\cup A_n$ is dense in $A$, this is enough to show that
$A$ locally embeds in $\{C_n\}$.
\end{proof}
\begin{rem}\label{ju17}
 By definition of the LLP
it is easy to see that any $C$ with LLP locally embeds in $\C$.
In particular, for any nuclear $C^*$-algebra $D$, $D \otimes_{\min}\C$ 
has LLP and hence locally embeds in $\C$.
\end{rem}
%  For the sake of notational clarity, in the rest of this section  we 
%  denote by $ \{ C_k\}$ ($k\ge 0$) the family such that
%   $$\forall k\ge 0\quad C_k=C.$$
Let $p_{[n,\infty]} $ denote 
the  $*$-homomorphism taking $x\in L$ to $(x_n, x_{n+1}, \cdots) \in \ell_\infty(\{ C_k, k\in [n,\infty] \} )$. 
Note that $p_{[n,\infty]} (Z)$ is a $C^*$-subalgebra of $\ell_\infty(\{ C_k, k\in [n,\infty] \} )$. \\
It is convenient to record here the following elementary fact.
\begin{lem}\label{deco}
Let $n\ge 1$. The map
$\Phi_n: z\mapsto (z_0,\cdots,z_{n-1} ) \oplus p_{[n,\infty]} (z) $ is an isomorphism
from $Z$ to $C_0\oplus\cdots\oplus C_{n-1} \oplus p_{[n,\infty]} (Z) \subset C_0\oplus\cdots\oplus C_{n-1} \oplus \ell_\infty(\{ C_k, k\in [n,\infty] \} )$.
\end{lem}
\begin{proof} It is clear that $\Phi_n$ is an injective $*$-homomorphism
with range included in $C_0\oplus\cdots\oplus C_{n-1} \oplus p_{[n,\infty]} (Z)$.
Let  $y^0\in C_0\oplus\cdots\oplus C_{n-1} $ and let $y\in p_{[n,\infty]}
(\cup Z_m)$.  We may assume $y\in p_{[n,\infty]}
(Z_m)$ for some $m >n$. Thus $y=(y_{n}, \cdots, y_m, T_{m+1}y_m,T_{m+2}T_{m+1}y_m, \cdots) $ with $y_m\in E_m$. Let
$$z=(y^0_0,\cdots,y^0_{n-1}, y_{n}, \cdots, y_m, T_{m+1}y_m,T_{m+2}T_{m+1}y_m, \cdots). $$
Then $z\in Z_m $ and $\Phi_n(z) =(y^0 , y)\in  C_0\oplus\cdots\oplus C_{n-1} \oplus p_{[n,\infty]} (Z)$.
This shows that the range of $\Phi_n$ is dense in $C_0\oplus\cdots\oplus C_{n-1} \oplus p_{[n,\infty]} (Z)$ and hence that $\Phi_n$ is surjective.
   \end{proof}
   
\section{Approximately multiplicative self-adjoint liftings of $F$}\label{am}
                
                Consider a separable $C^*$-algebra $C$ with LP and
                a $*$-homomorphism $$F : C \to \cl L(C)=
                \ell_\infty(C)/c_0(C).$$
                We know that $F$ admits
                a   c.c. lifting $(f_n): C \to \ell_\infty(C)$
                which is (automatically) asymptotically
               multiplicative and self-adjoint, meaning that
                          \begin{equation}\label{ju2}
                \forall x,y\in C\qquad \limsup\|f_n(xy)-f_n(x)f_n(y)\|=0  \text{   and    } f_n(x^*)=f_n(x)^*. \end{equation}
                Let $(\d_n)$ be a given positive sequence.
                We will construct a system
                of f.d. subspaces $E_n \subset  C$  ($n\ge 0$),  maps
                $$T_{n } : E_{n-1} \to E_{n},$$ a  sequence of integers $(m(n))$     and   a sequence $\vp_n>0$,  
                such that $\sum \vp_n\le 1$, satisfying the following properties:

                (i) Each $T_n$ is the restriction of   $f_{m(n)}$ to $E_{n-1}$ and $T_n$ is an $\vp_n$-morphism.

                (ii) $T_{n }(E_{n-1}) + T_{n }(E_{n-1})T_{n }(E_{n-1}) \subset E_{n } $.
                
                (iii)    For any $n> k$ we have $e^2 \sum_{k+1}^n  \vp_j< \vp_k.$

                (iv) For any $C^*$-algebra $D$ and any $\vp_n$ -morphism $\t : E_{n} \to D$
                we have $$\| [ Id_\C \otimes \t T_n] (t^{E_{n-1}} )  \|_{\max} \le 1+ \d_n .$$ 
                See Proposition \ref{ppchy} for the definition of $t^{E_{n-1}}$.
                By \eqref{tE},  (iv)  is equivalent to:
                
                (iv)' For any $C^*$-algebra $D$ and any $ \vp_n$-morphism $\t : E_{n} \to D$
                we have $$\|   \t T_n  \|_{MB(E_{n-1},D)} \le 1+ \d_n. $$

More explicitly by \eqref{tE'} (iv)' is equivalent to :
 
 (iv)'' For any $C^*$-algebra $D$ and any $  \vp_n$-morphism $\t : E_{n } \to D$
we have for any $C^*$-algebra $B$
         $$\| Id\otimes   \t T_n : B\otimes_{\max}  E_{n-1}  \to B\otimes_{\max} D  \|  \le 1+ \d_n. $$ 

Since $(f_n)$ is c.c. we have 
\begin{equation}\label{ech4} \| T_n\|_{cb} \le 1.  \end{equation}

\begin{rem}\label{r1} Note that (iv) applied to the identity map from $E_n$ to $C$
implies $\| [ Id \otimes  T_n] (t^{E_{n-1}} )  \|_{\max} \le 1+ \d_n .$
\end{rem}

The construction of our system is by induction: once we have $E_n,T_n, \vp_n$ we must produce $E_{n+1}, \vp_{n+1}$,  
and $T_{n+1} : E_n \to E_{n+1}$. Compared with \cite{155} the novelty is the use of  conditions (iii) and
(iv).

\begin{lem} \label{ch1} Let $C$ be a  $C^*$-algebra  with LP. Let $E\subset C$ be a f.d. subspace.\\
Then for any
           linear map $\psi:   E \to D$ ($D$ any other $C^*$-algebra) we have
           $$ \|\psi  \|_{mb} \le     \|( Id_\C \otimes \psi)( t^E)   \|_{\C\otimes_{\max} D} .$$
           In other words if we set $\hat \psi=Id_\C \otimes \psi : \C\otimes_{\max} E \to \C\otimes_{\max} D$
           we have $\|\hat\psi\| =\|\hat\psi( t^E)   \|$.
\end{lem}
 \begin{proof}  
            Since $C$ has the LP   (by Th. 10.5 in \cite{157}) $E$ is ``$\max$-controlable" in the sense of \cite{160}, i.e. there is
            $t^E$ in the unit ball of $\C\otimes_{\max} E$ such that  
            for any   $t \in B_{\C\otimes_{\max} E}$  there is a unital $*$-homomorphism
            $\pi: \C \to \C$ such that $t=(\pi \otimes Id_E) (t^E)$. Then $\|( Id_\C \otimes \psi)( t)\|_{\max}=  \|( \pi \otimes Id) (Id \otimes  \psi)( t^E) \|_{\max} \le  \|\hat\psi( t^E)   \|_{\max}$ whence the announced inequality.
             \end{proof}

           \begin{lem}\label{ch2} Let $C,{C_1}$ be $C^*$-algebras. 
           Let $E\subset C$, $F \subset {{C_1}}$ be  f.d.  subspaces. 
           For any $\d>0$ there is $\vp>0$ and 
             a f.d.s.a. superspace  $\cl E$ with  $E\subset \cl E\subset C$ 
           such that for any 
           $\vp$-morphism $\psi: \cl E \to {D}$ (${D}$ any other $C^*$-algebra) we have
           $$\forall x\in F \otimes E\quad \|(Id_{C_1} \otimes \psi)(x)   \|_{ {C_1}\otimes_{\max} {D} }\le (1+\d) \|x\|_{{C_1} \otimes_{\max} C}.$$
           \end{lem}
           
           \begin{proof} Let $I=\{ (\cl E,\vp)\}$ be the directed set of pairs with $E\subset \cl E$, $\vp>0$. Fix $\d>0$.
           It suffices to show  that there is $(\cl E,\vp)\in I$ such that 
           for all ${D}$, all
           $\vp$-morphisms $\psi: \cl E \to {D}$   and all   $x\in F \otimes E$ we have
           $ \|(  Id_{C_1} \otimes \psi)(x)   \|_{{C_1} \otimes_{\max} {D} }\le (1+\d) \|x\|_{{C_1} \otimes_{\max} C}$.\\
           Let us first fix $x\in F\otimes E$.
           For each $\alpha\in I$ with $E\subset \cl E_\alpha$,
           there is a ${D}_\alpha$ and
           an $\vp_\alpha$-morphism $\psi_\alpha: \cl E_\alpha \to {D}_\alpha$,
           such that
           $$\|(Id_{{C_1}}  \otimes \psi_\alpha)(x)\|_{{{C_1}}  \otimes_{\max} {D}_\alpha} \ge (1+\d)^{-1} \sup \|(
           Id_{{C_1}}  \otimes \psi )(x)\|_{ {{C_1}}  \otimes_{\max} {D}}$$ 
           where the last supremum runs over all ${D}$ and all
           $\vp_\alpha$-morphism $\psi: \cl E_\alpha \to {D} $.
           Note that this last supremum is finite since the nuclear norm
           of each $\psi:\cl E_\alpha \to {D} $ is at most
           $(1+\vp_\a) \dim(\cl E_\alpha) $.
         Let  $\psi'_\alpha : C \to {D}_\alpha$ be the map that extends $\psi_\alpha$   by $0$ outside $\cl E_\alpha$ (we could use a linear map but this is not needed at this point).
         Consider $\psi'=(\psi'_\alpha): C \to  \ell_\infty (I; \{{D}_\alpha\})$ and 
         let $Q: \ell_\infty (I; \{{D}_\alpha\})\to \ell_\infty (I; \{{D}_\alpha\})/c_0(I; \{{D}_\alpha\})$ be the quotient map.
         Then $\pi=Q \psi': C \to \ell_\infty (I; \{{D}_\alpha\})/c_0(I; \{{D}_\alpha\})$ is clearly a contractive $*$-homomorphism.
         We have a contractive morphism $Id_{{C_1}}   \otimes \pi : $ 
         $$ {{C_1}} \otimes_{\max} C  \to {{C_1}} \otimes_{\max} [\ell_\infty (I; \{{D}_\alpha\})/c_0(I; \{{D}_\alpha\})]  = [  {{C_1}} \otimes_{\max} \ell_\infty (I; \{{D}_\alpha\})]/
         [   {{C_1}}\otimes_{\max} c_0(I; \{{D}_\alpha\}) ]$$
         where the last $=$ holds by the ``exactness" of the max-tensor product (see e.g. \cite[p. 285]{P4}).
         Moreover, we have clearly a contractive morphism 
         $$  [  {{C_1}} \otimes_{\max} \ell_\infty (I; \{{D}_\alpha\})]/[{{C_1}}  \otimes_{\max} c_0(I; \{{D}_\alpha\})]
         \to  [\ell_\infty (I; \{ {{C_1}}  \otimes_{\max}  {D}_\alpha\}) ]/[c_0(I; \{ {{C_1}} \otimes_{\max} {D}_\alpha\}) ]. $$
         Since   $\pi(e)= Q ( (\psi_\alpha(e))_\a )$ for all $e\in E$ and since $x\in {{C_1}}  \otimes E$, it follows that 
         $$\limsup\nl_\alpha  \|(Id_{{C_1}}   \otimes \psi_\alpha)(x)\|_{{{C_1}}  \otimes_{\max} {D}_\alpha } \le \|x\|_{{{C_1}} \otimes_{\max} C},$$
         which proves the desired result for each fixed given $x\in F\otimes E\subset  {{C_1}}  \otimes E$.
         But since $F \otimes E$ is a finite dimensional subspace of ${{{C_1}}  \otimes_{\max} C } $
         we may replace the unit ball by a finite $\d$-net in it 
         (or invoke Ascoli's theorem). We can deal with the latter case by
        enlarging $\cl E$ finitely many times. We obtain the announced result (possibly with $2\d$ instead of $\d$).\\
         A different proof can be obtained using the Blecher-Paulsen factorization, as   in \cite[Th. 26.8]{P4}.
             \end{proof}
             
 \begin{lem}\label{ch3} In the situation of the preceding Lemma \ref{ch1},   
           for any $\d>0$ there is $\vp>0$ and 
             a f.d.s.a. superspace  $\cl E$ with  $E\subset \cl E\subset C$ 
           such that for any  other $C^*$-algebra $D$ and any
           $\vp$-morphism $\psi: \cl E \to D$   we have
           $$\forall x\in \C \otimes  E \quad \|( Id_\C \otimes \psi)(x)   \|_{\C\otimes_{\max} D}\le (1+\d) \|x\|_{\C\otimes_{\max} E},$$
           or equivalently for any $D$
           $$ \|\psi_{|   E} \|_{MB(E,D)} \le 1+\d.$$
           \end{lem}
           
            \begin{proof}  
            By Lemma \ref{ch2} (actually just by the first part of the proof of it)  applied with $x=t^E \in \C \otimes E$ 
            we find $\cl E$ such that for any $\vp$-morphism $\psi: \cl E \to D$ we have
            $ \|( Id_{\C} \otimes  \psi)(t^E) \| _{ \C\otimes_{\max} D} \le 1+\d.$ Therefore by Lemma \ref{ch1}
            $ \| \psi_E\|_{mb} \le 1+\d $.
                         \end{proof}

\begin{lem}\label{ch4} 
Consider a separable $C^*$-algebra $C$ 
with LP
 and a $*$-homomorphism $$F : C \to \cl L(C)=
                \ell_\infty(C)/c_0(C)$$
                together with   a 
                  c.c. lifting $(f_n): C \to \ell_\infty(C)$,
                which  is (automatically) asymptotically
                $*$-multiplicative.
                 Let $E\subset C$ be a f.d.s.a. subspace.
           For any $\d>0$ there is $\vp'>0$
           such that for any $m$ large enough
           there is
          a f.d.s.a. subspace $E_1\subset C$ 
          containing $  {f_m}(E)+{f_m}(E){f_m}(E)   $
          % if we set  
          %so that $T(E)+T(E)T(E)  \subset E_1$ 
          %satisfying the following:\\
          %we have $ E_1\supset {f_m}(E)+{f_m}(E){f_m}(E)  $
          %and
          so that  if we set $T = {f_m}_{|  E}: E \to E_1$
         %  then $T$
            %there is
          %a f.d.s.a. subspace $E_1\subset C$  
           %such that if we set $T={f_m}_{| E}$ 
        %
        %
           %such that $T(E)T(E)  \subset E_1$ 
          % and
         then  for any  other $C^*$-algebra $D$ and any
           $\vp'$-morphism $\psi: E_1 \to D$   we have
           \begin{equation}  \label{ju1}  \|\psi T \|_{MB(E,D)} \le 1+\d,\end{equation}
           and $T$ itself is  an $\vp'$-morphism.
         %  In particular for $\psi$ equal to the inclusion $E_1\to C$ we have
        %   $$ \| T \|_{MB(E,C)} \le 1+\d.$
           \end{lem}
           
           It will be convenient to set for $T: E \to E_1$
          \begin{equation}  \label{ju10} \|T\|_{mb,\vp} =\sup \|\psi T \|_{MB(E,D)} \end{equation}
           where the sup runs over all $C^*$-algebras $D$ and all $\vp$-morphisms $\psi: E_1 \to D$, so that  \eqref{ju1} could be written as $ \|T\|_{mb,\vp'} \le 1+\d $. Note   that  the case when 
           $\psi : E_1\to C$ acts as the identity yields $ \| T \|_{MB(E,C)} \le \|T\|_{mb,\vp'}$ for any $\vp'>0$.

            \begin{proof}  Let $\cl E$ and $\vp$ be as in Lemma \ref{ch3}. Fix $0<\vp'<\vp$ to be specified.
            Since $(f_m)$ satisfies \eqref{ju2},
            for all $m$   large enough
            ${f_m}_{|\cl E}: \cl E \to C$ is an $\vp'$-morphism. 
            A fortiori so is ${f_m}_{| E}:  E \to C$ and hence $T$.
             We set 
              $E_1= {f_m}(\cl E) + {f_m}(\cl E){f_m}(\cl E)$.
            Note  $T(E)  + T(E)T(E)  \subset E_1$.
            By
           Remark   \ref{rch1}
 for any $D$ and any
           $\vp'$-morphism $\psi: E_1 \to D$,             the composition $\psi'=\psi {f_m} _{|\cl E}$ is
           an $\vp''$-morphism on $\cl E$ where 
           $\vp'' =(1+\vp') \vp' +\vp'(1+\vp')^2$. 
           Thus we can clearly select and fix
           $\vp'\approx \vp/2$ so that $\vp'' \le \vp$ or even, say, $\vp'' =\vp$. Then  Lemma \ref{ch3}
           tells us that $ \|\psi T \|_{MB(E,D)}=\|\psi'_{|   E} \|_{MB(E,D)} \le 1+\d.$
             \end{proof}

Let $(\vp_n)$ and $(\d_n)$ be positive sequences such that
 \begin{equation}  \label{ech0} 
 \sum \vp_n<\infty \text{   and   } \d_n \to 0.
 \end{equation}
From what precedes we see that
to any embedding $C \to \cl L(C)$ we can associate (at least if $C$ has LP)
an inductive system and a resulting $A$ :

\begin{pro}\label{pch2} In the situation of Lemma
\ref{ch4}  for each positive sequence $(\d_n)$
                there is a summable positive sequence  $(\vp_n)$  
               and   an admissible inductive system $(E_n,T_n)$ with a sequence
                $(m(n))$  as before
                satisfying (i)-(iv)'.\\
 Moreover, if $(f_n)$ satisfies \eqref{cond}
then we can find $(E_n)$ and $m(n)$ satisfying
\eqref{ech5} in addition to (i)-(iv)'. 
\end{pro}

\begin{proof} By induction, suppose we have obtained the system up to
$ T_n : E_{n-1} \to E_n$ and $\vp_n$ with $T_n=f_{m(n)} $ restricted to $E_{n-1}$.
We then apply Lemma \ref{ch4} to $E=E_n$ with $\d=\d_{n+1}$ to produce
a number $\vp'$ so that choosing $m$ large enough   we may  set $\vp_{n+1}=\vp'$, $m(n+1)=m$,
$E_{n+1}=E_1$ and $T_{n+1}=  { f_{m(n+1)} }_{|E_n}$. Lemma \ref{ch4}
then implies
$$\|  T_{n+1}  \|_{mb,\vp_{n+1}} \le 1+\d_{n+1}.$$
As for condition (iii) since the inequality is strict we can clearly adjust further
$\vp_{n+1}<\vp'$ so that the condition remains valid.
Thus we obtain (i)-(iv)' for the next level $n+1$, which proves the
first part. 
If in addition $(f_n)$ satisfies \eqref{cond}, 
we can choose $m(n+1)$ large enough so that
 $\|{ f^{-1}_{m(n+1)} }_{|f_{m(n+1)}( E_n)} \|_{cb} <1+ \vp_{n+1}$,
 and this gives us \eqref{ech5} at the next induction step.
\end{proof}

\begin{rem}\label{rpch2'} 
We will use (iii) and (iv)' together as follows.
Consider $\t_n : E_n \to Z$. By Lemma \ref{rch1'} we know that  $\t_n$ is a $\d$-morphism
for $\d =e^2 \sum_{j>n} \vp_j$, and hence by (iii) an $\vp_n$-morphism.
By (iv)'   we have
$\|   \t_n T_n  \|_{MB(E_{n-1},Z)} \le 1+ \d_n. $
We claim that this implies
\begin{equation}  \label{ju6-} \forall n\ge 1 \quad \|\t_{n-1}\|_{MB(E_{n-1},Z)} \le 1+\d_n 
\end{equation} 
and hence also 
\begin{equation}  \label{ju6} \forall n\ge 0\quad \|\t_n\|_{MB(E_{n},Z)} \le 1+\d_{n+1} .\end{equation} 
Indeed, since $\forall x\in E_{n-1}$ 
$$\t_{n-1}(x)=(0,\cdots,0, x, w_n(T_n(x)) )=(0,\cdots,0, x, 0,\cdots)+\t_n(T_n(x)) ) $$ (recall \eqref{ju3} and \eqref{ju4})
we have
\begin{equation}  \label{ju5}
\Phi_n \t_{n-1}(x) =(0,\cdots,0,x) \oplus p_{[n,\infty]} \t_n T_n  (x)\end{equation} 
By Lemma \ref{deco} we have
 $\|\t_{n-1}\|_{MB(E_{n-1},Z)} \le \max\{ 1, \| p_{[n,\infty]}\t_n T_n \|_{MB(E_{n-1},   p_{[n,\infty]}(Z))} \}$,
 and hence
$$\|\t_{n-1}\|_{MB(E_{n-1},Z)} \le \max\{ 1, \| \t_n T_n \|_{MB(E_{n-1},   Z)} \}.$$
Whence our claim $\|\t_{n-1}\|_{MB(E_{n-1},Z)} \le 1+\d_n$ for all $n\ge 1$.
\end{rem}

Following the terminology of \cite{154} we say
that the inclusion of a subalgebra $Z$  in a larger one $  L$
is $\max$-injective if the associated $*$-homomorphism
$D\otimes_{\max } Z \to D\otimes_{\max }   L$ is injective (and hence isometric) for all $C^*$-algebras $D$.
For that to hold it suffices that it holds for $D=\C$.

  \begin{pro}\label{pch2'} Consider the situation of Proposition \ref{pch2}.
Assume $(\d_n)$ bounded. Then the associated $C^*$-algebra $Z$ has the LP and the inclusion $Z\to \ell_\infty(C)$ is $\max$-injective.
\end{pro}
\begin{proof}
To show that $Z$ has the LP we use the criterion from \cite{157}.
Let $(C_i)$ be an arbitrary family of $C^*$-algebras (we could take simply $C_i=\C$ for all $i$).
We claim that 
there is a constant $c$ such that for any $t\in Z \otimes \ell_\infty(\{C_i\}) $
the associated family $(t_i)$ with $t_i\in Z \otimes C_i $ satisfies
\begin{equation}  \label{jr} \|t\|_{Z \otimes_{\max} \ell_\infty(\{C_i\})} \le c \sup\nl_i\|t_i\|_{Z \otimes_{\max} C_i}  .
\end{equation} 
Actually, since here we are comparing $C^*$-norms if there is such a $c$ then $c=1$ works, and we obtain the criterion as in \cite{157}.
To verify this claim it suffices to check this
for $t\in Z_{n} \otimes \ell_\infty(\{C_i\}) $ with $c$ independent of $n$.
The following argument, simpler than the original one, is due to Jean Roydor.
The claim is an easy consequence of the fact that there is a linear isomorphism
$v_n: E_n \to Z_n$ (defined by $v_n(x)=\t_n(x)$) such that (recall \eqref{jr1})
\begin{equation}  \label{jr2}\|v_n\|_{MB(E_n,Z_n)} \|v^{-1}_n\|_{MB(Z_n,E_n)}\le 1+\d_{n+1}.\end{equation} 
Indeed, this follows by combining \eqref{ju6} with the observation that
$v^{-1}_n : Z_n \to E_n$ is nothing but the restriction of the $n$-th coordinate   $p_n: \ell_\infty(C) \to C$ which is 
a $*$-homomorphism   from $Z$ to $C$, and hence has mb-norm at most $1$.\\
Now, it is just a simple exercise to deduce that
 \eqref{jr} holds for $t\in Z_{n} \otimes \ell_\infty(\{C_i\}) $ with $c=1+\d_{n+1}$ from the fact that the analogous inequality holds
 with $c=1$ for all $t\in E_{n} \otimes \ell_\infty(\{C_i\}) $ (the latter   since $C$ has the LP).
 
 Recall the notation $L=\ell_\infty(C)$.  To show that $Z\subset L$ is $\max$-injective, let
$t\in \C \otimes Z$.
We will show that $\|t\|_{\C \otimes_{\max} Z}\le c \|t\|_{\C \otimes_{\max} L} $ for some constant $c$.
As in the first part we must have $c=1$ since both sides of the inequality are $C^*$-norms.
By perturbation we may clearly assume that
$t\in \C \otimes Z_{n}$. Then since $x= v_n p_n x$  for any $x\in Z_n$ we have by \eqref{ju6} again
$$\|t\|_{\C \otimes_{\max} Z_n} \le \|Id \otimes v_n\|_{MB(E_n,Z_n)} \| Id  \otimes p_n (t)  \|_{\C \otimes E_{n}} 
\le 
(1+\d_{n+1}) \|t\|_{\C \otimes_{\max} L}    
 $$
 where at the last step we used that $\|p_n\|_{MB(L,C)} \le 1$.
 This yields the announced result with $c=1+\sup_n   \d_{n}$.
 \end{proof}
\begin{rem}  As pointed out to me by Jean Roydor  the preceding argument highlights
the following general fact : if a separable $C^*$-algebra $Z$   locally  embeds  
in the mb-sense described by \eqref{jr2} into another one with the LP
then $Z$ inherits the LP.
Thus we may state: a separable $Z$ has the LP if and only if it embeds  in $\C$ 
locally in the mb-sense (meaning the same as in Definition \ref{jr4} but with mb in place of cb).
This is similar to the equivalence observed in \cite{161}   of the mb-versions
of the LLP and the LP.  Note that the terms are slightly misleading: it is not clear that
a $C^*$-subalgebra locally  embeds  
in the mb-sense, unless the inclusion is max-injective.
%GILLES what is a good example of subalgebra of $\C$ without LLP or LP ?
\end{rem}
 
\begin{pro} \label{pch3} Consider the situation of Proposition \ref{pch2}.
Assume $\d_n $ bounded. 
Assume in addition
that the   inductive system satisfies \eqref{ech5}, 
then the associated $C^*$-algebra $A$ has the LLP.
\end{pro}
\begin{proof} We set $\B=B(\ell_2)$. We claim that there is a constant $c$
for which the inclusion $A_n\to A$
satisfies $$\|\B \otimes_{\min} A_n \to \B \otimes_{\max} A\| \le c\ \  \text{  for all  }n  .$$ This clearly implies that the pair $(\B,A)$ is nuclear, i.e. that $A$ has the LLP. 

As in the proof of Proposition \ref{pch10} let 
$$\a_n=Q\t_n: E_n \to A_n.$$
On the one hand by
 \eqref{cho4} we have 
  \begin{equation}  \label{ju8} \|\a_n^{-1}:A_n \to E_n\|_{cb}\le 1+\eta_{n+1},\end{equation} which means
$$ \| Id \otimes \a_n^{-1} :  \B \otimes_{\min} A_n \to \B \otimes_{\min} E_n\|\le
1+\eta_{n+1}.$$ On the other hand
by \eqref{ju6}   we have
 \begin{equation}  \label{ju9}\|\a_n\|_{mb} \le 1+\d_{n+1},\end{equation}  which implies a fortiori
$$ \| Id \otimes \a_n :  \B \otimes_{\max} E_n \to \B \otimes_{\max} A\|\le
1+\d_{n+1}.$$
Recall   that here  the norm of $\B \otimes_{\max} E_n$ is the one induced by $\B \otimes_{\max} C$. Since 
 $C$ has the LP and a fortiori the LLP we have
 $\B \otimes_{\min} E_n= \B \otimes_{\max} E_n$ isometrically and hence we conclude that the inclusion $i_{A_n} : A_n \to A$ satisfies
 $$ \| Id \otimes i_{A_n} :  \B \otimes_{\min} A_n \to \B \otimes_{\max} A\|\le
(1+\eta_{n+1})(1+\d_n)(1+\d_{n+1}).$$ 
This proves the claim.
\end{proof}

\begin{rem}
 We should emphasize  here a subtle point about \eqref{ju10} and  (iv)' that we already used:
 a priori  $\t_{n}$ is an $\vp_{n}$-morphism from $E_{n}  $ to
 $L$ but its values being inside $Z$ we may view it as 
 an $\vp_{n}$-morphism into $Z$, and  hence (iv)' 
implies
 $\| \t_{n}T_{n} \|_{MB(E_{n-1}, Z)} \le 1+\d_{n}$. Similarly
  \eqref{ju9} means 
  \begin{equation}  \label{cho3'}
    \| \a_n \|_{MB(E_{n}, A)}
  \le 1+\d_{n+1} . \end{equation} 
  In this light the norm in \eqref{ju10} appears much stronger than
  the mb-norm of $T$. Specifically, it dominates the  mb-norm of $T$
  even when $T$ is viewed as acting into the $C^*$-algebra generated by
  its range, which a priori may augment the mb-norm.
 \end{rem}

\begin{thm}\label{ty}
Consider the situation of Proposition \ref{pch3}. Let $D$ be a separable $C^*$-algebra.
Assume $\d_n $ bounded and \eqref{cond}.  If
  the embedding $F : C \to \cl L(C)=
                \ell_\infty(C)/c_0(C)$
  is $D$-nuclear (i.e.
$Id_D \otimes F$ is bounded from $D\otimes_{\min} C $ to $D\otimes_{\max} \cl L(C)$), the   
admissible inductive system can be constructed so that
the  associated $C^*$-algebra $A$  
(as in Proposition \ref{pch1}) is $D$-nuclear  and locally embeds in $C$.
\end{thm}
    \begin{proof}  
The underlying idea is that the inclusion
$A_{n-1} \subset A_n$ is ``similar" to the map $T_n: E_{n-1} \to E_{n}$. 
By what precedes we have a factorization of  
$A_{n-1} \subset A_n$ as :

$$A_{n-1} {\buildrel {\alpha^{-1}_{n-1}} \over \longrightarrow  } E_{n-1} {\buildrel {T_n} \over \longrightarrow  }  E_n 
{\buildrel {\a_{n} } \over \longrightarrow  } A_{n}.
 $$
We already saw in \eqref{ju8} and \eqref{ju9}
that $\|{\alpha^{-1}_{n-1}}\|_{cb} \le 1+\eta_n$ and 
 $\| \a_n \|_{mb}
  \le 1+\d_{n+1} .$
 %SOMETHING SEEMED WRONG HERE, FIXED?
  We now turn to the assertion about the $D$-nuclearity of $A$.
 Recall that by the projectivity of the max-tensor product we have a canonical
 contractive $*$-homomorphism
$D\otimes_{\max} \cl L(C)\to  \cl L( D \otimes_{\max} C)$, and hence
for any $x\in D \otimes C$ our assumption on $F$   implies
\begin{equation}  \label{che4}\limsup\nl_n\| [Id \otimes f_n](x) \|_ {D \otimes_{\max} C}  \le
\| [Id_D \otimes F](x) \|_ {D \otimes_{\max} {\cl L} (C)}
\le \|x\|_{D \otimes_{\min} C}.\end{equation}

Let $D_n\subset D$ be an increasing sequence of f.d. subspaces
with dense union.
We indicate how the induction step should be modified to obtain
the assertion in the theorem. 
We ensure at each step
that $T_n :E_{n-1} \to E_{n}$ is such that
\begin{equation}  \label{che4'}\| Id\otimes T_n : D_{n-1} \otimes_{\min} E_{n-1} \to D_{n-1} \otimes_{\max} E_{n} \|\le 1+\d_n.\end{equation}
Suppose we have constructed $E_{n-1},E_n,T_n$. Using a suitable  fine enough finite net 
in the (compact) unit ball of $D_{n} \otimes_{\min} E_{n}$ and using 
\eqref{che4} 
we can find
$m$ (that will be our $m(n+1)$) such that, in addition to the properties already
imposed in the proof of Proposition \ref{pch2}, we have
$$\| Id\otimes f_m : D_{n} \otimes_{\min} E_{n} \to D_{n} \otimes_{\max} C\|\le 1+\d_{n+1}$$
Then we argue as before: we let $T_{n+1}=f_m$
and $E_{n+1}=T_{n+1}(E_n)+T_{n+1}(E_n)T_{n+1}(E_n)$, and we obtain \eqref{che4'}
with $n$ replaced by $n+1$, which is
 the next step of the induction. 
We will now show how  \eqref{che4'} for all $n\ge 1$  implies that $A$ is $D$-nuclear.
Using \eqref{cho3'} and 
\eqref{che4'} we find
$$\|Id \otimes \a_nT_n : D_{n-1} \otimes_{\max} E_{n-1} \to D_{n-1} \otimes_{\max} A_n   \|
\le  (1+\d_{n})(1+\d_{n+1})    . $$
Since
$\|{\alpha^{-1}_{n-1}}\|_{cb} \le 1+\eta_n$, the inclusion map $ i_{A_{n-1}}:  A_{n-1} \to A$, 
which is equal to $\a_nT_n a_{n-1}^{-1} $, satisfies
$\|Id \otimes i_{A_{n-1}} : D_{n-1} \otimes_{\min} A_{n-1} \to D_{n-1} \otimes_{\max} A   \|\le  
(1+\eta_n) (1+\d_{n})(1+\d_{n+1}) ,$
or equivalently
$$\|Id \otimes i_{A_{n-1}} : D_{n-1} \otimes_{\min} A_{n-1} \to D \otimes_{\max} A   \|\le  (1+\eta_n) (1+\d_{n})(1+\d_{n+1}) ,$$
and since $(1+\eta_n) (1+\d_{n})(1+\d_{n+1}) $ is bounded and  $\cup [D_{n-1} \otimes  A_{n-1}]$  dense in
$D \otimes_{\min} A$    we conclude that    $A$ is $D$-nuclear.\\
By Proposition \ref{pch10}  we already know that $A$ locally embeds in $C$.
\end{proof}
\begin{rem}\label{ju14}
In the situation of Theorem \ref{ty} with $\d_n\to 0$,
 the conclusion can be  summarized 
as follows:
\\
The inclusion $A_{n-1} \to A_{n}$ can be factorized
as
$$  A_{n-1} {\buildrel {\alpha^{-1}_{n-1}} \over \longrightarrow  } E_{n-1} {\buildrel {T_n} \over \longrightarrow  }  E_n {\buildrel { \a_n} \over \longrightarrow  } A_{n} $$
where $\|{\alpha^{-1}_{n-1}}\|_{cb} \to 1$, $\|\a_n\|_{mb} \to 1$ and
$\|Id \otimes T_n : D_{n-1} \otimes_{\min} E_{n-1} \to D_{n-1} \otimes_{\max} E_n \|\to 1$ (where for the latter we use the norm induced by $D\otimes_{\max} C$).
\end{rem} 
\begin{rem} \label{ju15}
More generally given a sequence $(D(i))_{i\ge 1} $ of separable $C^*$-algebras
such that $F: C  \to \cl L (C )$  is $D(i)$-nuclear for each $i$, we claim that  
 our system can be adjusted so that $A$ is $D(i)$-nuclear for each $i$.\\
 Indeed, for each $i$,  let $\{D(i)_N\mid N\ge 1\}$ be an increasing sequence of f.d.
 subspaces with union dense  in $D(i)$. 
 Let $\{\D_n\mid n\ge 0\}$ be an enumeration of the collection $\{D(i)_N\mid i\ge 1, N\ge 1\}$ where each element is repeated countably infinitely many times.
 Fix a subspace $\D\in \{\D_n\mid n\ge 0\}$. Then $\D=\D_{n-1}$ for some $n\ge 1$ that
 we can choose as large as we wish. Clearly the inductive argument of the preceding proof can be modified
 to ensure that 
 $$\forall n\ge 1\quad \|Id \otimes i_{A_{n-1}} : \D_{n-1} \otimes_{\min} A_{n-1} \to \D_{n-1} \otimes_{\max} A   \|\le  (1+\eta_n)(1+\d_{n})^2.$$
Thus $\D=\D_{n-1}$ implies
  \begin{equation}\label{ju13} \|Id \otimes i_{A_{n-1}} : \D  \otimes_{\min} A_{n-1} \to \D  \otimes_{\max} A   \|\le  (1+\eta_n)(1+\d_{n})^2.\end{equation}
Let $c=\sup (1+\eta_n)(1+\d_{n})^2.$ Since \eqref{ju13} holds for infinitely many $n$'s
the min and max norms are $c$-equivalent on $\D \otimes[\cup A_n]$, and since the latter holds for any $\D$, for each $i$, the same $c$-equivalence holds on 
$[\cup _N D(i)_N] \otimes [\cup_{n-1} A_{n-1}]$.
Since the latter is dense in $D(i) \otimes_{\min} A$ this proves our claim.

Unfortunately we do not see how to remove the assumption that the family
$(D(i))_{i\ge 1} $  is countable.
\end{rem}

  \begin{rem} In our original approach in \cite{155} it suffices to assume that
          $C$ has the LLP. We then use for each $n\ge 1$ the existence of a c.c. lifting for
          $F_{|E_{n-1}}$ and obtain a non-nuclear WEP $C^*$-algebra $A$ that locally embeds
          in $\C$. Then the LLP of $A$ follows from the general fact that WEP $\Rightarrow$
          LLP for algebras that locally embed in $\C$. 
          A priori the non-separability of $\B$ prevents us from taking $D=\B$ to prove that
          $A$ is $\B$-nuclear i.e. has the LLP, but once we know that $A$ locally embeds
          in $\C$,  it suffices for the LLP of $A$ to know that 
          $A$ is $j$-nuclear where  $j: \C \to \B$ is any embedding and where
          by $A$ is $j$-nuclear we mean that $j\otimes Id_A$ is bounded
          from $\C \otimes_{\min} A \to \B \otimes_{\max} A$ (see the appendix for a proof). With this in mind
          the separability of $\C$ allows us to use the same argument as for   separable $D$'s to check the LLP of $A$.
          \end{rem}

\begin{rem}
As already mentioned  in \cite{155}, we can perform our construction of inductive systems
in a more general setting that we will just briefly sketch here.
We give ourselves a sequence of isometric $*$-homomorphisms
$F(N) : C_N \to \cl L (C_{N+1})$ together with c.c.s.a. liftings
$f(N): C_N \to L (C_{N+1})$ so that if $Q(N+1): L (C_{N+1}) \to  \cl L (C_{N+1})$
denotes the quotient map we have $Q(N+1)f(N)=F(N)$.
Then if we assume all the $C_N$'s have the LP and that   the coordinates of each $f(N)$
are   asymptotically 
                 locally almost completely isometric  (i.e. they satisfy the condition \eqref{cond})
then we can produce an inductive system
with $E_{n-1} \subset C_{n-1}$ and with $T_n=f(n-1)_{m(n)}$ on $E_{n-1}$, and we obtain analogous properties.

\end{rem}

\section{Cone algebras}\label{ca}

Proposition \ref{pch1}
tells us how to associate a $C^*$-algebra $A$ to a $*$-embedding $F: C \to \cl L(C) $, equipped with 
a completely contractive lifting $f: C \to L(C)$.
By Theorem \ref{ty}  if $f$
 satisfies \eqref{cond}, if
  $C$ has the LP and if $F$ is $\C$-nuclear, the algebra $A$ has WEP and LLP. To valuably apply this result, we need to  exhibit   $F$'s with the required properties. Their existence  is not so immediate but    cone algebras
  provide us with  that crucial missing ingredient.

Let  $C_0=C((0,1])$  and  $C_1=C([0,1])$.  For  any  $C^*$-algebra  $A$,
  we  denote  by  $C_0(A)=  C_0  \otimes_{\min}  A$  the  so-called  cone  algebra
 of $A$ and by $C(A)$ its unitization. (Warning: $C(A)\not=C_1  \otimes_{\min}  A$ !)\\
    When  dealing  with  a  bounded mapping  $u:  A  \to  B$  between  $C^*$-algebras
  (or  operator  spaces)  we  will  denote  by  $$u_0:  C_0(A)  \to  C_0(B)$$
  the  bounded map  extending    $Id_{C_0}  \otimes  u$.

   The algebra $C$ considered throughout \S \ref{am}
   will now be replaced by the cone algebra $C_0({\bf C})$ or $C({\bf C})$.
    We use the bold letter ${\bf C}$ for a $C^*$-algebra to avoid a notational confusion.

Let $q: {\bf C} \to B$ be a surjective $*$-homomorphism and let $\I=\ker(q)$.
We will use freely the natural identifications
$$ B\simeq  {\bf C}/\cl I \quad \text{  and  } \quad  C_0(B)\simeq C_0 ({\bf C}/\cl I) \simeq C_0({\bf C})/C_0(\cl I), \quad C(B)\simeq C ({\bf C}/\cl I) \simeq C({\bf C})/C_0(\cl I).$$
Let  $(\sigma_n)$  be  a  quasicentral  approximate  unit  in  $\I$.
Our  construction  is simpler  if   there is such a 
  $(\sigma_n)$   (and hence $(1-\sigma_n)$)  formed  of  projections.
Then  the  mappings  $x\mapsto  (1-\sigma_n ) x$  are  approximatively  multiplicative. In general this does not exist.  However,
if we replace ${\bf C},B,q$ by $C_0({\bf C}),C_0(B),q_0$  then  it does exist
and the construction of the preceding section can be applied.

This trick of passing to cone algebras is
closely   related  to    Kirchberg's  \cite[\S  5]{Kir},
but    we  learnt it from  \cite[Lemma  13.4.4]{BO}.  A similar idea already
appears in \cite[Lemma 10]{CoH} for the suspension algebra in the context of
approximatively  multiplicative families indexed by a continuous parameter in $(0,\infty)$, but our goals seem unrelated.

 \begin{lem}\label{ll7}
            Let $q: {\bf C} \to B$ be a surjective $*$-homomorphism between $C^*$-algebras, with separable kernel $\cl I$.
            There is a pointwise  bounded sequence of   maps
             $\rho_n: C_0\otimes {\bf C} \to C_0\otimes {\bf C}$ that induces in the limit a  $*$-homomorphism
            $$\rho_\infty: C_0({\bf C}) \to \L(C_0({\bf C}) ) $$
          that  vanishes on $C_0(\cl I)$ and for which the
            associated map
            $$  \rho_{[\infty]} : C_0({\bf C})/C_0(\cl I) \to \L(C_0({\bf C}) ) $$
         (which satisfies $\rho_{[\infty]} q_0=\rho_\infty$)   is an isometric $*$-homomorphism.\\
            Moreover, we have $q_0\rho_n(x)=q_0(x)$ for all $x\in C_0\otimes {\bf C} $ and   all $n$.
              \end{lem}
           We could refer to the proof of Lemma 6.1 in \cite{155} (itself based on \cite[Lemma 13.4.4]{BO}),
           but we give the details for the reader's convenience.
           
           \begin{proof} Recall $\cl I=\ker(q)$.
            Let $(\sigma_n)$ be a net forming a quasicentral approximate unit of $\cl I$. 
            This means $\sigma_n\ge0$, $\|\sigma_n\| \le 1$,
            $\|\sigma_n x-x\|\to 0$ for any $x\in \cl I$ and
            $\|\sigma_n c-c\sigma_n\|\to 0$ for any $c\in {\bf C}$.
            We identify $C_0({\bf C})=C_0\otimes_{\min} {\bf C}$ with the set of ${\bf C}$ valued functions
            $f: [0,1] \to {\bf C}$ such that $f(0)=0$. 
            The set of polynomials
            $\cl P_0={\rm span}[t^m\mid m>0]$ is dense in $C_0$.
            Let $\rho_n: C_0 \otimes {\bf C} \to C_0({\bf C})$ be the map taking 
            $t\mapsto f(t) c$ ($f\in C_0,c\in {\bf C}$) to
            $t\mapsto f(t(1-\sigma_n)) c$. For instance (monomials)
            $\rho_n$ takes $t\mapsto t^m c$ to
            $t\mapsto t^m (1-\sigma_n)^m c$.
            Note $\|1-\sigma_n\|\le 1$ therefore for any $f\in C_0$ the function
            $t\mapsto f(t(1-\sigma_n))$ is in $C_0({\bf C})$
            with norm $\le \|f\|_{C_0 }$. This shows that
            $\sup_n\|\rho_n(y)\|<\infty$ for any $y\in C_0\otimes {\bf C}$, so that $(\rho_n)$ defines
            a map $\rho: C_0 \otimes {\bf C} \to \ell_\infty(C_0({\bf C}))$.\\
            We should warn the reader that  a priori we do not know
            that $\rho$ is bounded on $C_0\otimes_{\min} {\bf C}=C_0({\bf C})$, but we will show that it becomes contractive after composition with the quotient map
         $Q: \ell_\infty(C_0({\bf C})) \to \cl L(C_0({\bf C}))$.
           
             Since $\sigma_n\in \cl I$
             we have for all $x\in \cl P_0\otimes {\bf C}$ (and hence all $x\in \cl C_0\otimes {\bf C}$)
              \begin{equation}\label{us8}   q_0\rho_n (x) =q_0(x) . \end{equation}
              Indeed,   
            \eqref{us8} clearly holds
            in the case $x=t^m \otimes c$ and hence it holds by density
            for any $x\in C_0\otimes {\bf C}$.

            Since $\sigma_n$ is quasicentral we have
            $$\forall x,y\in C_0\otimes {\bf C}\quad 
            \limsup\nl_n\|
            \rho_n(xy)-\rho_n(x)\rho_n(y)\|=0.$$
            Indeed this reduces to the case of monomials which is obvious,
            and also $$\limsup\nl_n\|\rho_n(x)^*-\rho_n(x^*)\|=0.$$
            It follows that after composing $\rho$ by the quotient map
            $Q: \ell_\infty(C_0({\bf C})) \to \cl L(C_0({\bf C}))$ we obtain
            a  map 
            $\rho_\infty : C_0 \otimes {\bf C} \to \cl L(C_0({\bf C}))$ which
            is a $*$-homomorphism, such that $\rho_\infty=Q\rho$ on $C_0\otimes {\bf C}$.
            Since $C_0$ is nuclear, this  extends to a   $*$-homomorphism
        $\rho_\infty: C_0({\bf C}) \to \cl L(C_0({\bf C})) $  defined on the whole of 
            $C_0({\bf C})$ with $\| \rho_\infty \|\le 1$.
             
     We have
             \begin{equation}\label{us9}\forall f \in C_0\otimes \cl I\quad \limsup\nl_n\|
            \rho_n(f)\|=0 ,\end{equation}
            and hence  $\rho_\infty( C_0\otimes \I)=Q\rho ( C_0\otimes \I)=\{0\}$.
            Therefore,  after passing to the quotient by $C_0 (\cl I)  \subset \ker(\rho_\infty)$, we derive from $\rho_\infty$
            a  (contractive) $*$-homomorphism 
            $$\rho_{[\infty]} : C_0({\bf C})/C_0(\cl I) \to \L(C_0({\bf C}) )  .$$
          
          We have clearly $\rho_{[\infty]} q_0=Q\rho$ on $C_0\otimes {\bf C}$.
          To check that $\rho_{[\infty]}$ is isometric it suffices to show that
          for any $x\in C_0\otimes {\bf C}$ we have 
          $\|\rho_{[\infty]} q_0(x)\|\ge \|q_0(x)\|.$
            By \eqref{us8}
           we have
            \begin{equation}  \label{chu1}
             \forall x\in C_0\otimes {\bf C} \ \forall n\qquad  \|q_0(x)\|  \le \|\rho_n(x)\| .
            \end{equation}
 
             This implies
             $\|q_0(x)\|\le \limsup\|\rho_n(x)\|  = \|Q\rho(x)\|=\|\rho_{[\infty]} q_0(x)\|. $
          Thus    $\rho_{[\infty]}$ is isometric.      \end{proof}

            \begin{rem}\label{rchi}[On unitizations] We start by   some standard
            background.
            Recall that the unitization of  a $C^*$-algebra $A$
            is the unique  unital $C^*$-algebra $\hat A \supset A$ such that any unital 
            $*$-homomorphism $u: A \to B$  into a unital $C^*$-algebra $B$ uniquely
            extends to a unital $*$-homomorphism. The vector space $\hat A$
            is spanned by $A$ and the unit of $\hat A$, and $A$ is an ideal 
            in $\hat A$ such that $\hat A/A\simeq \CC  $. The Banach spaces
            $\hat A$ and $A\oplus \CC$ are isomorphic with equivalent norms. 
            Given a linear map $u: A \to B$ from a $C^*$-algebra $A$ to a unital one, let
   $\hat u : \hat A \to  B$ denote its unital extension to the unitization of $A$.  
   When $u$ is a $*$-homomorphism (resp. an isometric one), so is $\hat u$.\\
            Recall that $C({\bf C})$ (resp. $C(B)$) denotes the unitization of $C_0({\bf C})$
            (resp. $C_0(B)$).
            Using the identifications $C_0(B)\simeq C_0({\bf C})/ C_0(\cl I) \simeq C_0({\bf C}/\cl I)$ we have
            an isometric $*$-homomorphism 
            $$ \rho_{[\infty]} : C_0({\bf C}/\cl I)\to \cl L (C_0({\bf C})).$$
            %=C_0({\bf C}/\cl I) \simeq C_0({\bf C})/C_0(\cl I)$). 
            
            Clearly
            $C_0({\cl I})$ can be viewed as a subalgebra (in fact an ideal) in   $C({\bf C})$.
            Using the characteristic property of unitizations,
            it is easy to verify that $C({\bf C})/ C_0(\cl I)$  can be identified with  $C(B)=C({\bf C}/\cl I)$.
       We have  a canonical isometric embedding 
             $$i_{\cl L}: {\cl L} (C_0({\bf C})) \subset {\cl L} (C({\bf C})),$$ derived from the inclusion
         $C_0({\bf C}) \subset C({\bf C})$.
       Since  
        $\cl L (C({\bf C}))$ is unital
           we have a unital $*$-homomorphism
            $\tilde   \rho_{[\infty]} : C({\bf C}/\cl I) \to \cl L (C({\bf C}))$
           extending 
            $i_{\cl L}   \rho_{[\infty]}$.
           Clearly, $\tilde  \rho_{[\infty]}$ remains isometric.
                      Moreover,   
           denoting 
           $$V=\CC 1_{C({\bf C})} + C_0\otimes {\bf C} \subset C({\bf C})$$
       the map $\rho= (\rho_n): C_0\otimes {\bf C} \to L(C_0({\bf C}))\subset L(C({\bf C}))$  has a unital extension
             $\tilde \rho  : V \to L(C({\bf C}))$
           and $V$     is a dense subspace of $   C({\bf C})$. 
            Let $Q': L( C  ({\bf C})) \to \cl L( C  ({\bf C}))$ denote the quotient map.
            It is then easy to check that $Q'\tilde \rho $
            is the restriction to V of $\tilde \rho_{[\infty]}$.

           All in all this shows that there is a version
           of Lemma \ref{ll7} valid for the unital cone algebras as follows.       
      \end{rem}
            
            \begin{lem}\label{ll7u}
            Let $q: {\bf C} \to B$ be a surjective $*$-homomorphism between $C^*$-algebras ${\bf C} $ and $ B$, with separable kernel $\cl I$.
          Let $i: C_0({\bf C})/C_0(\cl I) \to  C({\bf C})/C_0(\cl I)$ be the natural inclusion. 
          Let $\tilde q_0 :  C({\bf C}) \to  C({\bf C})/C_0(\cl I) $ be the unital map
          associated to $iq_0$.
                      There is a pointwise  bounded sequence of  unital  maps
             $\tilde\rho_n: V \to V$ that induces in the limit a unital $*$-homomorphism
            $$\tilde\rho_\infty: C({\bf C}) \to \L(C({\bf C}) ) $$
          that  vanishes on $C_0(\cl I)$ and for which the
            associated map
            $$ \tilde \rho_{[\infty]} : C({\bf C})/C_0(\cl I) \to \L(C({\bf C}) ) $$
         (which satisfies $\tilde\rho_{[\infty]}\tilde q_0 =\tilde\rho_\infty$)   is a unital isometric $*$-homomorphism.\\
            Moreover, 
            we have $\tilde q_0\tilde \rho_n(x)=\tilde q_0(x)$ for all $x\in V $ and   all $n$.
              \end{lem}
            \begin{rem} Alternatively, we
            could denote the unitization of $\L(C_0({\bf C}) )$ by $\hat \L$ and by $\a: \L(C_0({\bf C}) ) \to\hat \L$
            the natural inclusion. Then  $\tilde \rho_{[\infty]} = \widehat{   \a\rho_{[\infty]} }\widehat{i_{\cl L}} $.
            %
             %have  $\tilde \rho_{[\infty]} = \widehat{   \rho_{[\infty]} }\hat\nu $  
            %where $\nu: \L(C_0({\bf C}) ) \to \L(C({\bf C}) )$ denotes the natural inclusion.
     \end{rem}
            Now let $j: {\bf C} \to B$ be an embedding  so that
            $j_0: C_0({\bf C}) \to C_0(B)$ and $F= \rho_{[\infty]}  j_0$
            are clearly also  embeddings. Let
            $\tilde{j_0} : C({\bf C}) \to C(B)$ be the canonical  extension of $j_0$ to the unitizations
            (again an isometric $*$-homomorphism). Recall   $B={\bf C}/  \cl I$.
             Using the identification
            $$C({\bf C})/ C_0(\cl I) =C(B) $$
            we may set $$\check F= \tilde \rho_{[\infty]} \tilde{j_0}  : C({\bf C}) \to {\cl L} (C({\bf C})).$$
                        Note that $\check F$ is but the canonical unital extension of the map $ i_{{\cl L} }F$. We refer the reader
                        to   the diagrams below.
                     
            \begin{lem}\label{ju18} Let $j$   be as above and let
         $$F : C_0({\bf C}) \to \cl L (C_0({\bf C}))
          \text{  (resp.   } \check F   : C({\bf C}) \to {\cl L} (C({\bf C}))) $$
         be as above.
         Let $f=(f_n)$ be a c.c. lifting for $F$  (resp. $\check F$).
            Then the condition \eqref{cond} holds.
             
                  \end{lem}
            \begin{proof}
            Let $V\subset \C$ be a f.d.  subspace and let
          $\D= \cl P_N \otimes V \subset C_0\otimes {\bf C}$ where 
          $\cl P_N={\rm span}[t,\cdots,t^N]$. Recall
          $jV \subset B ={\bf C}/\cl I$.
          Let $\psi:   jV \to {\bf C}$ be any linear lifting, i.e. we have
          $q\psi j(x)=j (x)$ for any $x\in V$, and hence
          $q_0\psi_0 j_0(x)=j_0(x)$ for any $x\in \D$. Let  
            $g_n={\rho_n\psi_0j_0}_{|\D}: \D \to C_0({\bf C})$ (actually $g_n$ takes values
            in $\cl P_N \otimes {\bf C}$). We first claim that $g_n$ is asymptotic
            to ${f_n}_{|\D}$ by which we mean
            $$\forall x\in \D\qquad \|g_n(x)-f_n(x)\|\to 0 . $$
            Indeed, we have
            $Q(f(x)) =F(x)= \rho_{[\infty]} j_0 (x)$
            and also $Q(g(x))=Q\rho\psi_0j_0(x)= \rho_\infty\psi_0j_0(x) =\rho_{[\infty]} q_0 \psi_0 j_0(x)= \rho_{[\infty]} j_0(x)$.  Thus
            $Q(f(x)) = Q(g(x))$ which is equivalent to our first claim.\\
            Our second claim is that
              $$
               \forall n\qquad \|{g^{-1}_n}_{|g_n(\D)} \|_{cb}\le 1.$$
          Indeed, since (see Lemma \ref{ll7}) $q_0 \rho_n=q_0$ on $C_0\otimes {\bf C}$ we have $q_0 g_n=q_0\psi_0j_0=j_0$, from which we derive
           $\|x\|=\|j_0(x)\|\le \|g_n(x)\|$.
           Equivalently  $
                \|{g^{-1}_n}_{|g_n(\D)} \| \le 1.$
               Similarly since $q_0$ is c.c. and $j_0$  completely isometric
          we obtain our second claim.\\
          To conclude,   an elementary perturbation argument
          shows that the two claims together with the complete contractivity of the $f_n$'s imply
           \begin{equation}  \label{chy4} \limsup\|{f^{-1}_n}_{|f_n(E)} \|_{cb}\le 1.\end{equation}
           Indeed, by \cite[Lemma 2.13.2 p. 69]{P4} for any $\vp>0$  and any $E\subset C_0({\bf C})$ there are 
 ${\D} \subset {\cl P} \otimes {\bf C}$  and a complete isomorphism $h: C_0({\bf C}) \to C_0({\bf C})$ such that
 $h({\D})=E$ and $\|h- Id\|_{cb} <\vp$ and also for convenience $\|h^{-1}- Id\|_{cb}<\vp$.
 Since ${\D}$ is f.d. $\limsup \|(f_n-g_n)_{|{\D}} \|_{cb} = 0$ by the first claim. By the two claims for any $0<\vp<1$ there is
           $N=N_\vp$ such that $\sup\nl_{n\ge N}\|(f_n-g_n)_{|{\D}} \|_{cb} \le \vp$ and also
           $\|{g^{-1}_n}_{|g_n({\D})} \|_{cb}\le (1-\vp)^{-1}$.
 Fix $n\ge N$.
 Let $x\in E$ with $\|x\|=1$. 
 Let $y= h^{-1} (x)\in {\D}$. 
 Then $1-\vp\le \|y\|\le 1+\vp$.
 We have $\|g_n(y)\|\ge (1-\vp)\|y\|   $ and $\|f_n(y)-f_n(x)\| \le \|y-x\| \le \vp $.
 It follows
 $$  \|f_n(x)\| \ge \|f_n(y)\|-\vp \ge \|g_n(y)\|-\vp  \|y\|-\vp\ge (1-\vp) \|y\| -\vp  \|y\|-\vp \ge   (1-2\vp)(1-\vp) -\vp.$$
 This shows that $\|{f^{-1}_n}_{|f_n(E)} \| \le [(1-2\vp)(1-\vp) -\vp]^{-1}$.\\
 An entirely similar reasoning with
 $x$ in the unit sphere of (say) $\B \otimes_{\min}  E$ shows the same upper bound
 for $ \|{f^{-1}_n}_{|f_n(E)} \|_{cb}$. Since this holds for any $n\ge N$ and any $\vp>0$   \eqref{chy4} alias \eqref{cond} follows. This completes the proof for $F$, the case of $\check F$ is identical.
   \end{proof}
   
   %%%%%%
               The situation is   summarized by the following diagrams.
            
            $$\xymatrix{&  C_0({\bf C})\ar[d]^{q_0}& L(C_0({\bf C}))\ar[d]^{Q} \\
C_0({\bf C})    \ar [r]\ar@/_1pc/[rr]_{F}  \ar [r]\ar@/^3pc/[urr]^{f}&  C_0(B) 
  \ar[r]   &  \cl L(C_0({\bf C}))} \qquad \xymatrix{&  C({\bf C})\ar[d]^{\tilde{q_0}}& L(C({\bf C}))\ar[d]^{Q'} \\
C({\bf C})    \ar [r]\ar@/_1pc/[rr]_{\check F}  \ar [r]\ar@/^3pc/[urr]^{f}&  C(B) 
  \ar[r]   &  \cl L(C({\bf C}))}$$
  
  By Remark \ref{rchi} we also have a unital isometric $*$-homomorphism   $\check F: C({\bf C}) \to \cl L (C({\bf C}))$
  that factors through $C(B)$.\\
          Intuitively, we would like to apply Lemma \ref{ll7} to the case of a surjective $*$-homomorphism 
          $q: C^*(\F) \to \B$ ($\F$ some large enough free group) together with an isometric $*$-homomorphism $j: C^*(\F) \to \B$, and to use the fact that $Id_\B$  is $\C$-nuclear  but
          since we wish to remain in the separable realm, we need to make some adjustments.
          There is clearly an embedding $\C \subset \B$, and it is easy to show that there is a separable $C^*$-subalgebra  $B\subset \B$ 
          with WEP that contains $\C$. Let $j: \C \to B$ be the inclusion map.
          Since $B$ is separable we also have a surjective $*$-homomorphism $q: \C \to B$, whence $F= \rho_{[\infty]}  j_0: C_0(\C) \to \cl L (C_0(\C))$.  
         % by \eqref{chy1}.
        The passage to the cone algebras is harmless:
        $C_0(\C)$ has the LP and $C_0(B) $ the WEP, i.e.  its identity map is $\C$-nuclear.
        A fortiori the map $F$ is $\C$-nuclear.
        
       Thus we obtain:        
        \begin{thm} \label{ty'}  There is an isometric $*$-homomorphism
        $F : C_0(\C) \to \cl L (C_0(\C))$
        (and a unital one  $\check F: C(\C) \to \cl L (C(\C))$) that is also $\C$-nuclear.
        \end{thm}

            %%%%%
   
                \begin{rem}[Unital variant]  
          Let $f=(f_n)$ be a unital c.c.  lifting of $\check F$. Then if we start
          by a unital subspace $E_0$, the construction produces unital algebras $Z$ and $A$ (see Remark \ref{roct}).
           \end{rem}
   
        We are thus in a position to apply Theorem \ref{ty}.
  
        \begin{thm} \label{ty'} In the situation just described
       there is an inductive system associated to the map $F : C_0(\C) \to \cl L (C_0(\C))$
       (resp.   $\check F: C(\C) \to \cl L (C(\C))$)
        for which the associated   (resp.  unital) $C^*$-algebra   $A$ has the WEP and the LLP
        but is not nuclear. Moreover, $A$ locally embeds in $\C$.
        \end{thm}
          \begin{proof}
          
          We apply Theorem \ref{ty} 
          with $C=C_0(\C)$ to the map $F= \rho_{[\infty]}  j_0: C_0(\C) \to \cl L (C_0(\C))$
          derived above from the coexistence of $j$ and $q$ (or of $j_0$ and $q_0$) on $\C$.
          By Lemma \ref{ju18}, the condition \eqref{cond} holds. By Proposition \ref{pch10} the algebra $A$ locally embeds in $C_0(\C)$ which itself
          locally embeds in $\C$ (see Remark \ref{ju17}). By Proposition \ref{pch3}
          $A$ has the LLP.
         We use $D=\C$ in Theorem \ref{ty}. Then $A$ is $\C$-nuclear which means it has the WEP.
         
         Let us now show how to make sure that $A$ is {\it not} nuclear.
         For that purpose we choose for $E_0$ a suitable subspace of $C_0(\C)$,
         for instance we can select $E_0$ linearly spanned by
         $(t\otimes U_1,\cdots , t \otimes U_d)$ where $(U_1,\cdots ,  U_d)$ are a $d$-tuple
         of free unitary generators of $\F_\infty$ viewed as unitaries in $\C$ (here $t$ means abusively the function $t\mapsto t$ on $[0,1]$). 
         It is known (see \cite[Th. 21.5, p. 336]{P4}) that for any
         subspace $S$ of a nuclear $C^*$-algebra we have
         $$d_{cb} (E_0, S) \ge d/2\sqrt{d-1}.$$
         When $d\ge 3$           we have $d/2\sqrt{d-1} >1.$
         Thus by \eqref{echy5}
          to show that $A$ is not nuclear with this choice of $E_0$  it suffices to choose
           our sequence $(\vp_n)$  so that $(1+\eta_1)^2 < d/2\sqrt{d-1}$.
          See Remark \ref{rchle} below for an alternate argument. 
   \end{proof}
     
           \begin{rem}[How to ensure that $\C$ locally embeds in $A$]\label{rchle} In the construction of the inductive system we may always enlarge
           $E_n$ at any step and then the next steps must of course be modified
           accordingly. Thus, if $S_n\subset C_0(\C)$ is any sequence of  f.d.s.a. subspaces
           that is dense for the $d_{cb}$ distance in the space of f.d.s.a. subspaces of $C_0(\C)$,
           we may define 
           $E_n= T_n(E_{n-1})+T_n(E_{n-1})T_n(E_{n-1}) +S_{n}$
           for each $n\ge 1$. Then
           \eqref{echy5} shows that we have a subspace $A'_n \subset A_n$
           corresponding to $S_n \subset  E_n$
           such that $d_{cb} (S_n,A'_n)\le (1+\eta_{n+1})^2.$
           If we use a sequence $(S_n)$ 
         where each of the original $S_n$'s
           is repeated infinitely many times, this shows that for the resulting system 
          $C_0(\C)$ and hence  $\C$ itself locally embeds
           in $A$.
        \end{rem}

        \section{Appendix}
        
        \begin{lem} Let $A$ be a $C^*$-algebra that locally embeds in $\C$.
        Then $A$ has the LLP if and only $A$ is $j$-nuclear, 
        where $j: \C \to B$ is any given isometric $*$-homomorphism from $\C$ to a
      $C^*$-algebra $B$ with   WEP.
        \end{lem}
        \begin{proof} The only if part is obvious. To show the if part, assume that
        $A$ is $j$-nuclear.
Let $u: A \to  C/\cl I$ be a c.p.c.c. map.
Since $Id_A$ is $j$-nuclear, the map $j \otimes u$
is a contraction from $\C \otimes_{\min} A \to B \otimes_{\max}C/\cl I$.
Let $E\subset A$ be a f.d. subspace. 
        Let $t\in E^*\otimes A$
        be the tensor associated to the inclusion map $E\to A$. 
        We have $\|t\|_{\min}=1$.         Recall that $E^*$ always locally embeds in $\C$ when $E$ does (see \cite[Lem. 21.10 p. 339]{P4}).
        Therefore there is   a c.i. embedding 
        $k: E^*\to \C $. We have $[k \otimes Id_A ](t) \in  \C \otimes A$ with $\|[k \otimes Id_A] (t)\|_{\min} =1$.
After composing with $j$, we find 
        $\|[jk \otimes Id_A](t) \|_{\max} \le 1$, and since c.p.c.c. maps between $C^*$-algebras preserve the max norm,
        we have $\|[jk \otimes u](t) \|_{\max} \le 1$.
        Here $t'=[jk \otimes u](t)  \in jk(E^*) \otimes_{\max} C/\cl I \subset  B\otimes_{\max} C/\cl I$.
        By the projectivity of the max-tensor product,
      there is an element $s$ in the unit ball of $ jk(E^*) \otimes_{\max} C \subset B \otimes_{\max} C$ lifting $t'$, meaning,
   denoting by $q: C \to C/\cl I$ the quotient map,
   that $t'=[Id \otimes q](s)$. A fortiori $\|s\|_{\min} \le 1$.
   Since $jk$ is c.i. we recover from $s$ an element $\hat t$  in the unit ball of $E^* \otimes_{\min} C $, such that $[Id \otimes q] (\hat t)
=t$. Let $v: E \to C$ be the linear map associated to $\hat t$.
Then $v$ lifts the inclusion $u_{|E}$ and $\|v\|_{cb}=1$.
This yields the LLP for $A$.
 \end{proof}
 
  \medskip
    
    \n\textbf{Acknowledgement.} I am grateful to Jean Roydor
    for useful conversations and the simpler proof of Proposition \ref{pch2'}.

\end{document}